\documentclass[12pt,a4paper]{amsart}

\usepackage{latexsym, amsmath, amscd, amssymb, amsthm}
\usepackage[T1,T2A]{fontenc}
\usepackage[utf8]{inputenc}
\newtheorem{te}{Theorem}
\newtheorem{pro}{Proposition}

\newtheorem{pr}{Example}
\newtheorem{lemma}{Lemma}
\newtheorem{rem}{Remark}

\DeclareMathOperator{\Der}{Der}
\DeclareMathOperator{\Ker}{Ker}
\DeclareMathOperator{\rk}{rk}
\DeclareMathOperator{\ord}{ord}

\begin{document}

\title{Centralizers of linear and locally nilpotent derivations}

%\author{L. Bedratyuk, Y. Chapovskyi, A. Petravchuk}
\author{L.~Bedratyuk}
\address{Department of Information Technology,
                Khmelnytsky National University,
                Khmelnytsky, Instytutska ,11,
                29016, Ukraine}
\email{LeonidBedratyuk@khmnu.edu.ua}
\author
{Y.Chapovskyi}

\address{Department of Algebra and Computer Mathematics, Faculty of Mechanics and Mathematics,
Taras Shevchenko National University of Kyiv, 64, Volodymyrska street, 01033  Kyiv, Ukraine}
\email{safemacc@gmail.com}
\author
{ A. Petravchuk}

\address{
Department of Algebra and Computer Mathematics, Faculty of Mechanics and Mathematics,
Taras Shevchenko National University of Kyiv, 64, Volodymyrska street, 01033  Kyiv, Ukraine}
\email{ apetrav@gmail.com , petravchuk@knu.ua}
\date{\today}

\begin{abstract}
     Let $K$ be an algebraically closed field of characteristic
     zero, $A = K[x_1,\dots,x_n]$ the polynomial ring,
     $R = K(x_1,\dots,x_n)$ the field of rational functions, and let
     $W_n(K) = \Der_{K}A$ be
     the Lie algebra of all $K$-derivations on $A$. If
     $D \in W_n(K),$ $D\not =0$ is linear (i.e. of the form
     $D = \sum_{i,j=1}^n a_{ij}x_j \frac{\partial}{\partial x_i}$)
     we give a description of the centralizer of $D$ in $W_n(K)$
     and point out an algorithm for finding generators of
     $C_{W_n(K)}(D)$ as a module over the ring of constants
     in case when $D$ is the basic Weitzenboeck derivation. In more
     general case when the ring $A$ is a finitely generated domain over
     $K$ and $D$ is a locally nilpotent derivation on $A,$ we prove that the centralizer
     $C_{{\rm Der}A}(D)$ is a "large" \ subalgebra in
     ${\rm Der}_{K} A$, namely     $\rk_A C_{\Der A}(D) :=
     \dim_R RC_{\Der A}(D)$
     equals ${\rm tr}.\deg_{K}R,$ where $R$ is the field of fraction of the  ring $A.$

\end{abstract}
\maketitle
\section{Introduction}
Let  $\mathbb{K}$ be an algebraically closed field of characteristic zero,
$A = \mathbb{K}[x_1,\dots,x_n]$ the polynomial algebra and
$R = \mathbb{K}(x_1,\dots,x_n)$ the field of rational functions in  $n$ variables. Recall that   a  $\mathbb{K}$-linear map $D:A \to A$  is called a $\mathbb{K}$-derivation of the algebra  $A$ if $D(fg) = D(f)g + fD(g)$ for any  $f,g \in A$. All $\mathbb K$-derivations of the $\mathbb{K}$-algebra  $A$ form the Lie algebra $W_n(\mathbb{K})$ over the field  $\mathbb{K}$ with respect to the commutation and every element $D\in W_n(\mathbb K)$  can be uniquely  written in the form
$D = f_1 \partial_1 + \dots + f_n \partial_n$, where
$f_i \in A$ і $\partial_i := \frac{\partial}{\partial x_i}$ are partial derivatives. Note that every derivation   of the algebra  $A$ can be uniquely extended on the field of fractions  $R=\mathbb K(x_1, \ldots , x_n),$ we will  denote this extension by  the same letter   when no confusion can arise.  The Lie algebra of all $\mathbb K$-derivations of the field $R$ will be denoted by ${\widetilde W}_n(\mathbb{K}),$ elements of this algebra can be uniquely written in  the form
$D=\varphi _1 \partial_1 + \dots + \varphi _n \partial_n, $ where $\varphi _i\in R, i=1, \ldots n.$

The structure of the Lie algebra  $W_n(\mathbb{K})$ is of great interest because, from the geometrical point of view,  $W_n(\mathbb{K})$ is  the Lie algebra of all vector fields on  $\mathbb{K}^n$ with polynomial coefficients; from the viewpoint of differential equations,  any  derivation $D=f_1 \partial_1 + \dots + f_n \partial_n$ can be considered as an  autonomous system of  ordinary differential equations
$$\dot{x} = F(x), \ {\rm where}  \ x=x(t)=(x_1(t), \ldots ,x_n(t)), \
F(x) = (f_1(x),\dots,f_n(x)).$$
For a given  derivation  $D \in W_n(\mathbb{K})$ one can consider its centralizer  $C_{W_n(\mathbb K)}(D)$ in $W_n(\mathbb  K)$, i.e the set of all elements of the Lie algebra $W_n(\mathbb{K})$ that commute with $D$; we will denote it briefly by $C_{W_n}(D)$). The centralizer  $C_{W_n}(D)$ is obviously a subalgebra of the Lie algebra  $W_n(\mathbb{K})$  and  $D\in C_{W_n}(D).$ The problem of describing centralizers of elements $D \in W_n(\mathbb{K})$ is closely related to some problems from theory of differential equations and geometry, it is solved only in some special cases (see for example, \cite{Petien, Nagloo,CEP,Finston, Panyu}).

 The paper is organized as follows. In Sect. 2, we consider centralizers  of linear derivations, i.e. derivations of the form
$$
D = \sum_{i=1}^n f_i \partial_i, \;
f_i = \sum_{j=1}^n a_{ij}x_j, \;
a_{ij} \in \mathbb{K}.
$$
Such a derivation is completely defined  by the matrix
$A = (a_{ij})_{i,j=1}^n$,  the commutator of derivations  $D$
і $D' = \sum_{i,j=1}^n b_{ij} x_j \partial_i$ with the matrix
$B = (b_{ij})$ is determined by the matrix  $[A,B]=AB-BA$. Therefore the Lie algebra of all linear derivations is isomorphic to the general linear Lie algebra
$\mathfrak{gl}_n(\mathbb{K}).$ It will cause no confusion if we use the same notation to these isomorphic  algebras  and  assume that  $\mathfrak{gl}_n(\mathbb{K})\subset W_n(\mathbb K).$

For a linear derivation  $D = \sum_{i,j=1}^n a_{ij}x_j \partial_i$ one can consider two centralizers:
$C_0 = C_{\mathfrak{gl}_n(\mathbb{K})}(D)$ and  $C = C_{W_n(\mathbb{K})}(D)$
(obviously, $C_0 \subseteq C$).
The structure of the centralizer  $C_0$ is well known, since $C_0$ consists of linear derivations determined by matrices that commute with $A = (a_{ij})$.
Description of such matrices was obtained many years ago (see, for example, \cite{Gantmacher}).
It turned out that in some sense the subalgebra  $C$ is near to  $C_0$.
This was proved in Theorem  \ref{Theorem1} which states that
$C_{W_n(\mathbb{K})}(D) =
(FC_{\mathfrak{gl}_n(\mathbb{K})}(D)) \cap W_n(\mathbb{K}),$
where $F$ is the kernel of  $D$ in $R = \mathbb{K}(x_1,\dots,x_n)$
 Unfortunately, this theorem does not allow in general to find  $C_{W_n}(D)$ since the intersection
$(FC_0) \cap W_n(\mathbb{K})$ is very hard to calculate.

   The third section of the paper provides  a method for finding generators of  $C=C_{W_n}(D)$ (as a module over ${\rm Ker}D$)  in the case when the matrix of $D$ is a single Jordan block (Theorem \ref{Th3}). To solve this problem we describe the kernels of  powers $D^n, n\geq 2$  of the Weitzenb$\ddot{\rm o}$ck derivation  under condition that we know the generators of  ${\rm Ker}D$  as a subalgebra of $A$ (Theorem  \ref{Th2}). Note that many papers were devoted to finding  generators of  kernels of locally nilpotent derivations (see, for example, \cite{Bedratyuk}, \cite{Miyan}, \cite{Now_Nag}).

The fourth section  of the paper deals with centralizers of locally nilpotent derivations  $D$ on finitely generated domains  $A$ over an arbitrary field $\mathbb{K}$ of characteristic zero.   In the case, when  $D$ has a slice in $A$, (i.e. there exists $s\in A$ such that  $D(s) = 1$) we prove that
$
C_{\Der_{\mathbb{K}}A}(D) \simeq C^{(-)} \rightthreetimes \Der C,
$
where  $C = \Ker D$ в $A$, $C^{(-)}$  is a vector space over $\mathbb  K$ with zero multiplication, i.e an abelian Lie algebra and elements of  $\Der_{\mathbb{K}}C$
act on  $C^{(-)}$ as derivations of the algebra  $C$ (Proposition   \ref{Proposition3}).
In general case, it was proved for a locally nilpotent derivation $D$ of a domain $A$ over $\mathbb K$ with
${\rm tr}.\deg_{\mathbb{K}}A = n$ that the centralizer  $C_{\Der_{\mathbb{K}} A}(D)$
has rank  $n$ over  $A$ (see Theorem \ref{Theorem3}). Note that centralizers of locally nilpotent derivations were studied in  \cite{Finston} from another point of view.

We use standard notation. If  $L$ is a subalgebra of the Lie algebra
$\Der _{\mathbb{K}}(A)$ and  $R$ the field of fractions of  $A$ then the rank of  $L$ over $A$ is the dimension
$\dim_R RL$; we denote it by $\rk_R L$.
 The index of a derivation  $D\in W_n(\mathbb K)$  is, by definition, the least number of nonzero coefficients in  $D = f_1 \partial_1 + \dots + f_n \partial_n.$
A derivation  $D\in  \Der _{\mathbb K}(A)$ is called locally nilpotent derivation  if for any  $a \in A$ there exists  $k=k(a)$ such that
$D^k(a)=0$. The smallest  natural  $k$ with  $D^k(a) \ne 0$, but
$D^{k+1}(a)=0$ will be called  $D$-order of the element  $a$ and denoted by $\ord_D(a)$.
For the sake of brevity, we will write $C_{W{n}}(D)$ i $ C_{\mathfrak{gl}_{n}}(D),$ without indicating the fixed ground field  $\mathbb K.$

\section{Centralizers of linear derivations}

  Throughout this section,  $A=\mathbb K[x_1, \ldots ,x_n]$, $W_n(\mathbb K)=\Der _{\mathbb K}(A).$

\begin{lemma} \label{l:aut}
     Let  $D \in W_n(\mathbb{K}), \;
          f \in \Ker D$ and let $\varphi$ be an arbitrary automorphism of the ring  $A$.
     Then   $\varphi^{-1}D\varphi \in W_n(\mathbb{K})$ and
     \begin{enumerate}
          \item $\Ker_A (\varphi^{-1}D\varphi) = \varphi^{-1}(\Ker D)$ ;
          \item $\varphi^{-1} (fD) \varphi =
                 \varphi^{-1}(f)\varphi^{-1}D\varphi$ ;
          \item $C_{W_n}(D) =
                 \varphi(C_{W_n}(\varphi ^{-1}D\varphi ))\varphi^{-1}$.
     \end{enumerate}
\end{lemma}
\begin{proof}
     The proof is straightforward.
\end{proof}

\begin{lemma} \label{l:com}
     Let
     $D = \sum_{i=1}^n f_i \partial_i, \;
     T = \sum_{i=1}^n g_i \partial_i \; f_i,\,g_i \in A$ be
      derivations from   $W_n(\mathbb{K})$. The derivations $T$ and   $D$
     commute   if and only if
     $D(g_i) = T(f_i),\; i=1,\dots,n$.
\end{lemma}
\begin{proof}
 It suffices to note that $g_i = T(x_i),\; f_i = D(x_i),\; i=1,\dots,n$.
\end{proof}

\begin{lemma} \label{l:decomp}
     Let  $D = \sum_{i=1}^n f_i \partial_i \in W_n(\mathbb{K})$ be such a derivation that
     $f_1,\dots,f_k \in \mathbb{K}[x_1,\dots,x_k]$ and
     $f_{k+1},\dots,f_n \in \mathbb{K}[x_{k+1},\dots,x_n]$.
     If
     $T \in C_{W_n}(D), \; T = \sum_{i=1}^n g_i \partial_i,$
     then the derivations
     $T_1 = \sum_{i=1}^k g_i \partial_i$ and      $T_2 = \sum_{i=k+1}^n g_i \partial_i$ belong to     $C_{W_n}(D)$ and $T = T_1 + T_2$.
\end{lemma}
\begin{proof}
     Put  $T_1 = \sum_{i=1}^n t_i \partial_i$, where
     $t_i = g_i, \; i=1,\dots,k,$ and $t_i = 0, \; i=k+1,\dots,n.$ Let us show that
      $[T_1,D] = 0$.   According to  Lemma \ref{l:com},
      it suffices to show that $T_1(f_i) = D(t_i), \; i=1,\dots,n$.
    Let first  $i=1,\dots,k$. Then  $t_i = g_i$ and
     $T_1(f_i) = T(f_i) = D(g_i) = D(t_i)$. Let now
     $i=k+1,\dots,n$. Then obviously  $T_1(f) = 0$   and,  by conditions of the lemma, we have
     $D(t_i) = D(0) = 0$. Thus,  $D(t_i) = T_1(f_i)$ for
     $i=1,\dots,n$ and by Lemma  \ref{l:com} we have $[D,T_1] = 0$.
     Similarly, one can show that $[D, T_2] = 0$.
\end{proof}

\begin{rem}\label{ratio}
One can easily check that the statements of Lemmas 2-3 are true for derivations of the field $R=\mathbb K(x_1, \ldots , x_n)$, i.e. elements of the Lie algebra ${\widetilde W}_n(\mathbb K).$

\end{rem}

\begin{te}\label{Theorem1}
    Let  $D = \sum_{i,j=1}^n a_{ij}x_j \partial_i \in W_n(\mathbb{K})$
     be a linear derivation, $F = \Ker _{R}D$ be a field of constants of  $D$
     in  $R=\mathbb K(x_1, \ldots ,x_n).$ If  $C_{\mathfrak{gl}_n}(D)$ and
     $C_{W_n}(D)$ are centralizers of $D$ in
     $\mathfrak{gl}_n(\mathbb{K}) \subset W_n(\mathbb{K})$ and in
     $W_n(\mathbb{K}),$ respectively, then  $C_{W_n}(D) = (FC_{\mathfrak{gl}_n}(D)) \cap W_n(\mathbb{K})$.
\end{te}
\begin{proof}
Without loss of
generality one can assume that the matrix $A = (a_{ij})$ of the derivation
   $D$  is written in  Jordan normal form
          \begin{equation} \label{JNF}
          A = \bigoplus_{i=1}^t J_{n_i-n_{i-1}}(\lambda_i) \quad
          \text{for some} \quad
          0=n_0 < n_1 < \dots < n_t = n, \;
          \lambda_1,\dots,\lambda_t \in \mathbb{K}.
     \end{equation}
    (we consider lower triangular Jordan blocks).  To prove this,   it is enough to apply   Lemma \ref{l:aut}.
     Further, since the sum   (\ref{JNF}) is direct, we have
     $D = D_1 + \dots + D_t$, where $D_i \in W_n(\mathbb{K}), i=1, \ldots t$ and is defined by the matrix
     $$
     O_{n_1} \oplus \dots \oplus  J_{n_i - n_{i-1}}(\lambda _{i})
     \oplus \dots \oplus O_{n_t - n_{t-1}}
     $$ with the zero matrices  $O_{n_{j}-n_{j-1}}, j\not =i.$		
     Let  $T \in C_{W_n}(D)$ be an arbitrary element of the centralizer.
     Then, by Lemma  \ref{l:decomp}, which can be easily generalized to  finitely many of terms, the  equality $T = T_1 + \dots + T_t$ holds, where the summands $T_i = \sum_{s=n_{i-1}+1}^{n_i} g_s \partial_s$ belong to  $C_{W_n}(D)$     (note that in  general case the  polynomials $g_j$ depend on all variables $x_1,\dots,x_n$).
     First, let us  show that
     $$
     T_1 = \sum_{s=1}^{n_1} g_s \partial_s \in
     (FC_{\mathfrak{gl}_n}(D))\cap W_n(\mathbb{K}).
     $$
     Here the derivation  $D_1$ is defined  by the matrix of the form  $$J_{n_1}(\lambda _1)\oplus O_{n_2-n_1}\oplus \cdots \oplus O_{n_t-n_{t-1}}, $$ where  $O_{n_i-n_{i-1}}, i=2, \ldots , t$ is  a square zero matrix of  suitable order, and  in our notation
     $$
     J_{n_1}(\lambda _1)=\left(
       \begin{array}{ccccc}
         \lambda _1 & 0 & \cdots & 0 & 0 \\
         1 & \lambda _1 & 0 & \cdots & 0 \\
         \cdots & \cdots & \cdots & \cdots & \cdots \\
         0 & \cdots & 1 & \lambda _1 & 0 \\
         0 & 0 & \cdots & 1 & \lambda _1 \\
       \end{array}
     \right).
          $$

      By Lemma  \ref{l:com}  the equality  $[D,T_1] = 0$ implies      $$
     D(g_1) = \lambda_1 g_1,  D(g_2) = \lambda_1 g_2+g_1,\dots,\; D(g_{n_1}) = \lambda_1 g_{n_1}+g_{n_1-1}.
     $$
     Let $r$  be the index of the derivation  $T_1$, i.e. the smallest natural number such that $g_r \ne 0$. Then  $g_{r-1}=0$ and we have  $D(g_r)=\lambda _1g_r.$
		     On the other hand, it is obvious that
     $$
     D(x_1) = \lambda_1 x_1, D(x_2) = \lambda_1 x_2 + x_{1}, \cdots ,
     D(x_{n_1}) = \lambda_1 x_{n_1} + x_{{n_1}-1}.
     $$
     Taking into account the mentioned above inequality
     $g_r \ne 0$, we obtain
     $g_r/x_{1}\not =0$ and  $D(g_r/x_{1})= 0$, i.e. $g_r/x_{1} \in F$,
     where  $F = \Ker _RD$ is  the field of constant of  $D$ in  $R$.

    Next, let us consider  the linear derivation
     $D_1 \in \mathfrak{gl}_n(\mathbb{K}) \subset W_n(\mathbb{K})$ that
    is defined by the matrix
     $$J_{n_1}(\lambda_1) \oplus O_{n_2-n_1}
     \oplus \dots \oplus O_{n_t - n_{t-1}},
     $$ where  $O_{n_i - n_{i-1}}$
     are square zero matrices of  suitable orders.
		
     Obviously, $[D_1,D] = 0$ and $D_1 = W_1 + V_1$, where the linear derivation
      $W_1$ is defined by the diagonal matrix
     $$d_{n_1}(\lambda_1) \oplus O_{n_2-n_1}
     \oplus \dots \oplus O_{n_t - n_{t-1}},
     $$ and  $V_1$ is given by the lower triangular matrix
         $M_1=J_{n_1}(0) \oplus O_{n_2-n_1} \oplus
     \dots \oplus O_{n_t - n_{t-1}}$.
     It is easy to see that $[W_1, D] = 0,$ and therefore also $[V_1, D]=0.$ One can directly verify that
      the linear derivation  $V_1^{(r-1)},$   defined by the matrix  $M_1^{r-1}$,
     is of the  form  $x_{1} \partial_r + \dots + x_{n_1-r+1} \partial_{n_1}$
     and  commutes wіth $D$.
     Then
     $$T_1^{(2)}:=(T_1 - \frac{g_r}{x_{1}}V_1^{(r-1)}) \in FC_{W_n}(D)\subset {\widetilde W}_n(\mathbb K) ,$$
     where $F$ is the field of constants of $D$ in the field $R$ (recall that $g_r/x_1\in F$).  The derivation $T_1^{(2)}$ obviously  has  a greater index than $T_1$ (recall that the index is the smallest number  of the non-zero
       coefficients of derivation). Repeating these considerations a sufficient number of times and  taking into account the Remark \ref{ratio} we  obtain a derivation $T_1^{(s)}, s\geq 2$ which commutes with       $D$ and is of the form
      $h_{n_1}\partial _{n_1}$ for some rational function  $h_{n_1}\in F\cdot A\subset R$.
			
 But then  $D(h_{n_1})=0,$ that is, the derivation $h_{n_1}\partial _{n_1}$ lies in  $FC_{\mathfrak {gl}_n}(D).$ It follows that the derivation $T_1$ is a sum of  some derivations of  $FC_{\mathfrak{gl}_n}(D),$  because at each step we subtract from $T_1$ a derivation  from  $FC_{\mathfrak{gl}_n}(D).$
     Therefore $T_1 \in (FC_{\mathfrak{gl}_n}(D))\cap W_n(\mathbb{K})$.
It can be similarly shown that
     $T_i \in (FC_{\mathfrak{gl}_n}(D))\cap W_n(\mathbb{K})$ for
     $i=2,\dots,t$. This completes the proof.
\end{proof}

%%%%%%%%%%%%%%%%%%%%%

\section{Centralizer of the basic Weitzenboeck derivation}

    Recall that a linear derivation $D$ on
$A = \mathbb{K}[x_1,\dots,x_n]$ is  called the basic Weitzenb$\ddot{\rm o}$ck derivation
if $D$ is defined by the matrix $J_n(0)$ consisting of one Jordan block
with the eigenvalue $\lambda=0$.
(we consider lower triangular Jordan blocks).
 This derivation can be written in coordinates as
$D = x_1 \partial_2 + \dots + x_{n-1} \partial_n$.
If   $T = \sum_{i=1}^n g_i \partial_i$ is an arbitrary derivation on $A$ that commutes with $D,$ then we have by Lemma \ref{l:com}
$$D(g_1) = 0,\; D(g_2) = g_1,\; \dots,\; D(g_n) = g_{n-1}.$$
For any polynomial $f \in A,$ let us denote by $D_f$
the  derivation on $A$  of the form
$$
D_f = D^{n-1}(f)\partial_1 + D^{n-2}(f)\partial_2 + \dots +D(f)\partial_{n-1}  + f \partial_n.
$$
There is a convenient criterion for verifying whether a derivation
$T \in W_n(\mathbb{K})$ commutes with the basic Weitzenb$\ddot{\rm o}$ck derivation.
\begin{pro} \label{p:repr}
A derivation $T \in W_n(\mathbb{K})$ commutes with the basic Weitzenb$\ddot{\rm o}$ck derivation $D$ if and only if
there is a polynomial $f \in A$ with $D$-order $\le n-1$ such that
$T = D_f$.
\end{pro}
\begin{proof}
Take any derivation  $T \in C_{W_n}(D)$ and write down it in the form   $T = \sum_{i=1}^n g_i \partial_i$
     for some polynomials $g_i \in A$. Then from the condition $[T,D] = 0$
      we obtain (by Lemma \ref{l:com}) the equalities
     $$D(g_1) = 0,\; D(g_2) = g_1,\; \dots,\; D(g_n) = g_{n-1}.$$
     Hence $g_n$ is a  polynomial of $D$-order $\le n-1$.
     Letting $f = g_n$ we have $T = D_f$ and $\ord_D(f) \le n-1$.
     The other way around, if $T = \sum_{i=1}^{n-1} D^{n-i}(f) \partial_i+f\partial _n$ and
     $\ord_D(f) \le n-1$, then again by Lemma \ref{l:com} we get that
     $[T,D] = 0$.
\end{proof}
The Weitzenb$\ddot{\rm o}$ck derivation $D$ on $A = \mathbb{K}[x_1,\dots,x_n]$
determines the filtration
 \begin{equation} \label{filtr}
 0 = A_0 \subset A_1 \subset \dots \subset A_n \subset \dots ,
 \end{equation}
 where $A_1 = \Ker D, \; A_i = \{f \in A  \ | \ D(f) \in A_{i-1}\},\;i=2,\dots
.$ It follows from  Proposition \ref{p:repr}  that for finding $C_{W_n}(D)$
it is sufficient to study the terms $A_i$ of the filtration   for $i \le n-1$, since
$T \in C_{W_n}(D)$ are of the form $T = \sum_{i=1}^n D^{n-i}(f)\partial_i$
for $f \in A_{n-1}$.
Note that $A_{n-1}$  is a module over the algebra of constants
$A_1 = \Ker D$ of the derivation $D$.

For studying $A_{n-1}$ we apply some ideas from the paper \cite{Bedratyuk}, where
the Weitzenb$\ddot{\rm o}$ck derivation $D$ was embedded in a $3$-dimensional Lie subalgebra
$L = \mathbb{K} \left< D, \widehat{D}, H \right> \subset W_n(\mathbb{K});$  the Lie algebra $L$ is constructed in such a way to be isomorphic to the simple Lie algebra $\mathfrak{sl}_2(\mathbb{K})$.

Linear derivations $\widehat{D}$ and $H$ on $A$ we define by the rule:
$$
\widehat{D}(x_i) = i(n-i)x_{i+1},\quad H(x_i) = (n-2i+1)x_i, i=1, \ldots ,n.
$$
It can be directly verified  (see \cite{Bedratyuk}) that the following commutation relations hold
\begin{equation} \label{com_rel}
[D,\widehat{D}] = H,\quad [H,D]=2D,\quad [H,\widehat{D}] = -2\widehat{D},
\end{equation}
that is,  $L \simeq \mathfrak{sl}_2(\mathbb{K})$.

The derivation $\widehat{D}$ is obviously locally nilpotent;
$\widehat{D}$-order of a polynomial $f \in A$ we will denote by $\ord(f)$
and call the order of $f$.
It is easy to see that each monomial
$x_1^{\alpha_1} \dots x_n^{\alpha_n} \in A$ is an eigenvector for
$H$ with the eigenvalue
$$\lambda _{(\alpha _1, \ldots ,\alpha _n)}=n \sum_{i=1}^n \alpha_i - (\alpha_1 + 3\alpha _2+ \cdots + (2i-1)\alpha _i+\cdots +(2n-1)\alpha_n).$$
We  call this eigenvalue the weight of the monomial  $x_1^{\alpha_1} \dots x_n^{\alpha_n}.$
A polynomial $f \in A$ will be called \emph{isobaric}, if all its monomials
have the same weight. Since all the monomials form a basis of the linear space
$A = \mathbb{K}[x_1,\dots,x_n]$ over $\mathbb{K}$,  the linear operator
$H$ is diagonalized in this basis.
But then the vector space $A$ is the direct sum $A = \bigoplus_{\lambda_ {(\alpha_1,\dots,\alpha_n)}} A_{\lambda_ {(\alpha_1,\dots,\alpha_n)}}$  of subspaces,
where $A_{\lambda_{(\alpha_1,\dots,\alpha_n)}}$ is the linear  span of all monomials
of weight $\lambda_{(\alpha_1,\dots,\alpha_n)}$, that is $A_{\lambda_{(\alpha_1,\dots,\alpha_n)}}$ consists of
isobaric monomials of the weight $\lambda_{(\alpha_1,\dots,\alpha_n)}$.

The statement of the next technical lemma follows immediately from the main properties of $sl_{2}(\mathbb K)$-modules
(see, for example, \cite{Humphreys}).

\begin{lemma}\label{sl}
Let $D$ be the basic Weitzenb$\ddot{o}$ck derivation on $\mathbb K[x_1, \ldots , x_n]$ and  let $f, g\in A$ be
isobaric polynomials with weights  $\lambda$ and $\mu , $ respectively. Then:

 (1) $D(f)$ and $\widehat{D}(f)$ are  isobaric polynomials of the weights $\lambda +2$ and $\lambda -2,$ respectively;

 (2) $fg$ is an isobaric polynomial with the weight $\lambda +\mu;$

 (3) If $f\in {\rm Ker}D$ and $f=f_{\lambda _1}+\cdots +f_{\lambda _k}$ is a
 decomposition of the polynomial $f$ into the sum of isobaric polynomials of pairwise distinct weights,
 then $f_{\lambda _i}\in {\rm Ker}D, i=1, \ldots ,k.$
\end{lemma}

It follows from the known Weitzenb$\ddot{\rm o}$ck theorem \cite{Weitz}  that the algebra of constants  ${\rm Ker} D$
of the Weitzenb$\ddot{\rm o}$ck derivation on $\mathbb K[x_1, \ldots , x_n]$ is finitely generated.
Increasing the number of generators (if needed) we can always assume (by Lemma \ref{sl}) that ${\rm Ker} D$
has a finite system of generators consisting of isobaric polynomials.

The following lemma indicates how the derivations $D$ and $\widehat D$ act on
the terms of the filtration
$0 = A_0 \subset A_1 \subset \dots \subset A_n \subset \dots $ (see (\ref{filtr})).

\begin{lemma}\label{isobar}
     Let $b \in A_k \setminus A_{k-1}$ be an isobaric polynomial,
     $k \ge 2$. Then
     \begin{enumerate}
          \item $\widehat{D}(D(b)) \in A_k \setminus A_{k-1}$;
          \item $D^{s-1}(\widehat{D}D(b)) = \widehat{D}D^s(b) + \lambda D^{s-1}(b)$
     for any $s \ge 2$ and some
     $\lambda = \lambda(s),\; \lambda \in \mathbb{K}$.
     \end{enumerate}
\end{lemma}
\begin{proof}
     1.
     Let $\deg b = t$ and let $V$ be  the set of all polynomials of $A$
     of degree at most $t$. Since $D, \widehat{D}, H=[D,\widehat{D}]$ are
     linear derivations on $A$, the subspace $V$ is  invariant under the natural action of the Lie algebra
     $L = \mathbb{K} \left< D, \widehat{D}, H \right> \simeq
     \mathfrak{sl}_2(\mathbb{K})$. Obviously $\dim _{\mathbb K}V < \infty$ and $b \in V$
    Then the  $L$-module $V$ can be  decomposed into  direct sum
     $V = V_1 \oplus \dots \oplus V_m$ of irreducible $L$-modules and
     $b = b_1 + \dots + b_m,\; b_i \in V_i$. Since
     $D^{k-1}(b) \ne 0,\; D^k(b) = 0$, at least one of the summands, for example
     $b_i,$ has the same property, that is
     $D^{k-1}(b_i) \ne 0,\; D^k(b_i) = 0$. Let us denote $e = D, f = \widehat{D}, h = H$ and choose in the irreducible
     $L$-module $V_i$ the standard basis
     $v_0,\dots, v_{n_i}$ (see, for example, \cite{Humphreys}). Then the element $b_i$ can be written in the form
     $b_i = \alpha_0 v_0 + \dots + \alpha_{k-1}v_{k-1}. $
     One can easily verify in these standard notations that $D^{k-1}(\widehat{D}D(b_i)) \ne 0$, but
     $D^k(\widehat{D}D(b_i)) = 0$. Taking into account our choice of the element $b_i$ we have
     $D^{k}(\widehat{D}D(b)) = 0$, $D^{k-1}(\widehat{D}D(b)) \ne 0$, that is
     $\widehat{D}D(b) \in A_k \setminus A_{k-1}$.

     2.
     Induction by $s$. If $s = 2$, then
     $$D \widehat{D}D(b) = \widehat{D}D^2(b) + HD(b) = \widehat{D}D^2(b) + \lambda D(b), \
     \lambda \in \mathbb{K},$$ since $D(b)$ is isobaric by Lemma \ref{sl}
     (recall that by  conditions of the lemma $b$ is an isobaric polynomial).
     Assume  the second statement of the lemma is true for $s-1$; we will  prove it for $s$.
     We have
     $$
     D^{s-1}(\widehat{D}D(b)) = D(D^{s-2}\widehat{D}D(b)) =
     D(\widehat{D}D^{s-1}(b) + \mu D^{s-2}(b))
     $$
     for some $\mu \in \mathbb{K}$ by the inductive hypothesis. Then
     \begin{align*}
     D^{s-1}(\widehat{D}D(b)) = D\widehat{D}D^{s-1}(b) + \mu D^{s-1}(b) =
     (\widehat{D}D + H)D^{s-1}(b) + \mu D^{s-1}(b) &= \\
     \widehat{D}D^s(b) + HD^{s-1}(b) + \mu D^{s-1}(b) =
     \widehat{D}D^s(b) + \alpha D^{s-1}(b) + \mu D^{s-1}(b) &= \\
     =\widehat{D}D^s(b) + (\alpha + \mu)D^{s-1}(b)
     \end{align*}
     (here we have used the fact that $D^{s-1}(b)$ is an isobaric polynomial
     with some eigenvalue $\alpha$ for $H$). This completes the proof.
\end{proof}
Next, let $a_1, \ldots , a_k$ be a system of generators of the kernel $A_1={\rm Ker}D$
(as a subalgebra of $A$). Without loss of generality one can assume  that the polynomials $a_1, \ldots , a_k$ are isobaric.
For convenience, we write  $ {\widehat D}^0(a_i)=1, i=1, \ldots k.$
Consider a system of subsets of $A$ of the form
$$ S_i =
     \{
          \widehat{D}^{k_1}(a_{i_1}) \times \cdots \times \widehat{D}^{k_t}(a_{i_t})\ | \
          a_{i_j} \in \{a_1,\dots,a_k\},\; k_{j}\geq 0, \  \sum_{j=1}^t k_j \leq i-1
     \}, \ i=1, 2,\ldots.$$

By our convention, we have  $S_1=\{ 1\},$ $S_2=\{ 1, {\widehat D}(a_1), \ldots ,{\widehat D}(a_k)\},$
and by Lemma \ref{sl} all the elements of the set $S_i, i\geq 1$ are isobaric polynomials.

\begin{te}\label{Th2}
     Let $D$ be the basic Weitzenb$\ddot{o}$ck derivation on
     $A = \mathbb{K}[x_1,\dots,x_n]$. Let us choose a system of isobaric generators
     $a_1,\dots,a_k$ of the kernel $A_1 = \Ker D$ (as a subalgebra in $A$)
     and denote $A_i = \Ker D^i,\; i \ge 1$. Then
     $A_i,\; i \ge 1$ is an $A_1$-module with the set of generators (as a module) of the form
     $$
     S_i =
     \{
          \widehat{D}^{k_1}(a_{i_1}) \times \cdots \times  \widehat{D}^{k_t}(a_{i_t}) \ \vert \
          a_{i_j} \in \{a_1,\dots,a_k\},\; \sum_{j=1}^t k_j \leq i-1
     \}.$$

\end{te}
\begin{proof}
     Induction by $i$. For $i=1$ everything is clear because   $A_1$ as $A_1$-module is generated by the element $1.$
     Assuming the statement of the theorem to hold for $i-1$, we prove it for $i$. Let
     $b \in A_i \setminus A_{i-1}$ be an arbitrary element.
     Then the element $b$ can be written in the form $b = b_1 + \dots + b_t$, where $b_s$ are isobaric polynomials
     of pairwise different weights, $b_s \in A_i,\; s=1,\dots, t,$ and at least
     one of the polynomials $b_s$ belongs to the set $A_i \setminus A_{i-1}$.
     Therefore, without loss of generality we can assume that $b$ is isobaric.

     By Lemma \ref{isobar}, it holds $\widehat{D}D(b) \in A_i \setminus A_{i-1}$ and
     by the same Lemma
     \begin{equation}\label{isoaction}
     D^{i-1}(\widehat{D}D(b)) = \widehat{D}D^i(b) + \lambda D^{i-1}(b)
     \end{equation}
     for some
     $\lambda =\lambda (i)\in \mathbb{K}$. Since $b \in A_i$, we have
     $D^i(b) = 0$; thus from the equality (\ref{isoaction}) we obtain that
     $D^{i-1}(\widehat{D}D(b)) = \lambda D^{i-1}(b)$ for some
     $\lambda \in \mathbb{K}.$  Additionally, by Lemma \ref{isobar}, $\lambda \ne 0$.
     Hence, we have the equality $D^{i-1}(\widehat{D}D(b) - \lambda b) = 0$. The latter means  that
     the element $d := \widehat{D}D(b) - \lambda b$ \ belongs to $ A_{i-1}$  and  $b = \lambda^{-1}(\widehat{D}D(b) - d).$
     By the inductive hypothesis $d$ is a linear combination of elements of the set
     $S_{i-1}$ with coefficients in $A_1 = \Ker D$. Besides,
     the element $D(b) \in A_{i-1}$ and consequently by inductive hypothesis
     $$D(b) = \sum_{s_j \in S_{i-1}}c_j s_j,\; c_j \in A_1.$$
     Applying $\widehat{D}$ to the both sides of the last equality we get
     \begin{equation}     \label{eqn}
     \widehat{D}D(b) = \sum_{s_j \in S_{i-1}}
     (\widehat{D}(c_j)s_j + c_j \widehat{D}(s_j)).
     \end{equation}
     Note that $\widehat{D}(s_j)$ is a linear combination of elements of
     $S_i$ with integral coefficients. In addition, for
     $c_j \in A_1 = \Ker D$ we have that $c_j = c_j(a_1,\dots,a_k)$ is a polynomial in the generators $a_1,\dots,a_k$ of the kernel $\Ker D = A_1$.
     It is easy to see that
     $$
     \widehat{D}(c_j) = \sum_{s=1}^k
     \frac{\partial c_j}{\partial a_s}(a_1,\dots,a_k) \widehat{D}(a_s),
     $$
     and hence $\widehat{D}(c_j)$ is a linear combination of elements of $S_2$
     with the coefficients in $A_1$. Therefore, from the equality (\ref{eqn}) we see that
     $\widehat{D}D(b)$ is a linear combination of elements of $S_i$ with
     the coefficients in $A_1 = \Ker D$. The latter means  that $S_i$ is a system of generators of
     $A_1$-module $A_i.$ The proof is completed.
\end{proof}
\begin{te}\label{Th3}
     Let $D$ be the basic Weitzenb$\ddot{o}$ck derivation on
     $A = \mathbb{K}[x_1,\dots,x_n]$,  let $a_1,\dots,a_k$ be a system of generators of
     the kernel $A_1 = \Ker D$ (as a subalgebra in $A$.)
     If
     $$
     S_n =
     \{
          \widehat{D}^{k_1}(a_{i_1}) \times \cdots \times \widehat{D}^{k_t}(a_{i_t}) \ \vert \
          a_{i_j} \in \{a_1,\dots,a_k\},\; \sum_{j=1}^t k_j \leq n-1
     \},
    $$
     then $C_{W_n(\mathbb{K})}(D)$   has a system of generators
     $\{D_{s} \ | \ s \in S_n\}$ (as a module over $\Ker D$), where $D_s = \sum_{i=1}^n D^{n-i}(s) \frac{\partial}{\partial x_i} .$
\end{te}
\begin{proof}
     For any element $s \in S_n$ we have $D^n(s) = 0$.
     Let $T$ be an arbitrary element of the centralizer $C_{W_n}(D)$.
     Then, by Proposition \ref{p:repr}, $T = D_b,\; b \in A_n.$
     By Theorem \ref{Th2}, the element $b$ is of the form $b = \sum_{s \in S_n} r_i s_i$, where $r_i \in \Ker D$.
     Therefore $ D_b = \sum_{s \in S_n} r_i D_{s_i}$. But then
     $\{ D_{s_i},\; s_i \in S_n\}$ is a system of generators of the $A_1$-module
     $C_{W_n}(D)$.
\end{proof}

\begin{pr} Let $D_3$ be the basic Weitzenb$\ddot{o}$ck derivation $D_3$
on the polynomial ring $\mathbb K[x_1, x_2, x_3].$
Let us build a system of generators for the centralizer $C_{W_3(\mathbb K)}(D_3)$
as a module over  ${\rm Ker}D_3.$

The kernel $D_3$  is generated by isobaric polynomials $a_1=x_1, \ a_2=x_1x_3-(1/2)x_2^2$
(see, for example, \cite{Bedratyuk}). Note that
$${\widehat D}(x_1)=x_2, {\widehat D}(x_1x_3-(1/2)x_2^2)=0, {\widehat D}^2(x_1)=x_3, {\widehat D}^2(x_1x_3-(1/2)x_2^2)=0.$$
Then the $A_1$-modules $A_1$, $A_2$, $A_3$  have the  systems of generators
$$S_1=\{ 1\}, \ S_2=\{ 1, \ x_2\}, \ S_3=\{ 1, \ x_2, \ 2x_3, \ 4x_2^2\}, \ {\rm respectively.}$$
By Theorem \ref{Th3}, we obnain a system of generators of the centralizer $C_{W_3(\mathbb K)}(D_3)$ as a $A_1$-module
$$ \partial _3, \ x_1\partial _2+x_2\partial _3,  \  2x_1\partial_1+2x_2\partial_2+2x_3\partial_3, \ 8x_1^2\partial_1+8x_1x_2\partial_2+4x_2^2\partial_3.$$

\end{pr}

\begin{rem}
The systems of generators, constructed in Theorems \ref{Th2} and \ref{Th3} are not necessarily minimal.
\end{rem}

\section{Centralizers of locally nilpotent derivations}
In this section, we assume that  $A$ is a finitely generated domain over an arbitrary field $\mathbb K$ of characteristic zero. The next lemma contains some known properties of locally nilpotent derivations on such domains.
\begin{lemma}[see, for example, \cite{Freudenburg}] \label{slice}
     Let  $D \ne 0$ be a locally nilpotent derivation on a finitely generated domain $A$ over a field  $\mathbb{K}$ of characteristic zero and $C={\rm Ker}_{A}(D).$  Then
     \begin{enumerate}
          \item $C$ is a subalgebra of $A$ that is invariant under any  $T \in \Der_{\mathbb{K}}A$ commuting  with $D$;
          \item if  $D$ has a slice  $s$ in $A$, i.e. there exists  $s\in A$ such that  $D(s) = 1$, then  $A$  is the polynomial ring in  $s$ over $C$, that is $A \simeq C[s]$;
          \item if  $D$ does not have any slices in $A$, then there exist
 $h_0 \in C,\; h_1 \in A$ such that  the element
$s=h_1/h_0$ is a slice for the extension  $\overline D$ of $D$ on the  ring  $A[h_0^{-1}]$
and the kernel  $\overline{D}$ in $A[h_0^{-1}]$ coincides with  $C[h_0^{-1}]$.

     \end{enumerate}
\end{lemma}
\begin{pro}\label{Proposition3}
     Let $D \ne 0$ be a locally nilpotent derivation of a domain  $A$ over $\mathbb K$ and  $C ={\rm Ker}D$ in  $A.$
     If  $D$ has a slice in $A,$ then
$C_{\Der_{\mathbb{K}} A}(D) \simeq C \rightthreetimes \Der_{\mathbb{K}}C,$
     where $C$ is considered as a Lie algebra over  $\mathbb{K}$ with zero multiplication and the action of  $\Der_{\mathbb{K}}C$ on $C$ is natural.
\end{pro}
\begin{proof}
     It follows from  Lemma \ref{slice} that $A \simeq C[s].$ Let $T$ be an arbitrary element of $C_{\Der_{\mathbb{K}} A}(D).$ Then the equality  $[T,D] = 0$ implies
      $T(s) \in C.$ Denote  $c_0 = T(s).$ It is easy to see that  $T - c_0 D \in C_{\Der_{\mathbb{K}}A}(D)$ and     $(T - c_0 D)(s) = 0$. Consider two subspaces $L$ and  $M$ of
     $C_{\Der_{\mathbb{K}} A}(D)$ of the form
     $$
     L = \{T \in C_{\Der_{\mathbb{K}} A}(D) \  \vert \ T(s) = 0\}, \
     M = \{T \in C_{\Der_{\mathbb{K}} A}(D) \  \vert \ T(C) = 0\}, {\rm respectively}.
     $$
     It is easy to check  that  $L$ is a subalgebra of  $C_{\Der_{\mathbb{K}} A}(D),$
      $M$ is an abelian ideal from $C_{\Der_{\mathbb{K}} A}(D)$  and   $M = C \cdot D.$ One can easily show that
     $C_{\Der_{\mathbb{K}}A}(D) \simeq  M \rightthreetimes L$ is a semidirect sum.  Note also that  $L \simeq \Der_{\mathbb{K}}C$ because any derivation of the algebra  $C$ over  $\mathbb{K}$ can be uniquely extended to a derivation $T$ of the algebra  $A$ with $T \in C_{\Der_{\mathbb{K}}A}(D)$
     (it suffices to put $T(s) = 0$).
\end{proof}
\begin{te}\label{Theorem3}
     Let  $A$ be a finitely generated domain over a field $\mathbb{K}$ of characteristic zero and  $D \ne 0$ be a locally nilpotent derivaton on  $A$.
     Then the centralizer $C_{\Der_{\mathbb{K}}A}(D)$ has rank  $n$ over $A,$ where
     $n = tr.\deg_{\mathbb{K}}A.$
\end{te}
\begin{proof}
     First, let the derivation $D$ have a slice  $s$ in $A$.
     It follows from Lemma \ref{slice} that  ${\rm tr.deg}_{\mathbb K}C = n-1.$
     Choose any transcendence basis  $c_1,\dots c_{n-1}$ of   $C$ (over the field   $\mathbb{K}$) and consider partial derivations $\frac{\partial}{\partial c_i},\; i=1,\dots,n-1$ on the subalgebra
     $\mathbb{K}[c_1,\dots,c_{n-1}] \subseteq C.$
     It is  known that these derivations can be uniquely extended on the field of fractions      $\mathrm{Frac}(C)$ (we use the same symbols  for extended
     $\frac{\partial}{\partial c_i}$ when no confusion can arise). Let $S = C\setminus\{0\}$ and let
     $A[S^{-1}]$ be the ring of fractions. Since  $A = C[s],$ we have   $A[S^{-1}] = \mathrm{Frac}(C)[s].$

     Further, let us extend the derivations
     $\frac{\partial}{\partial c_i},\; i=1,\dots,n-1$
     on $A[S^{-1}]$, putting
     $\frac{\partial}{\partial c_i}(s) = 0,\; i=1,\dots,n-1$.
     Let  $a_1,\dots,a_k$ be a system of generators of the algebra  $A.$  Denote
      $\frac{\partial}{\partial c_i}(a_j) = e_{ij}/d_{ij}$
     for some  $e_{ij} \in A, \, d_{ij} \in S = C \setminus \{0\},$ and
     put  $d = \prod_{i,j} d_{ij}.$ Then we obviously have $d \in C$ and
     $d \frac{\partial}{\partial c_i}(a_j) \in A,\;
     i=1,\dots,n-1,\, j=1,\dots,k.$. The latter means that
     $d\frac{\partial}{\partial c_i}(A) \subseteq A,\; i=1,\dots,n-1.$ It is easy to see that the derivations  $d \frac{\partial}{\partial c_i}$ commute with the derivation  $D$ on  $A$ (recall  that we extended  $\frac{\partial}{\partial c_i}$
     on $A[S^{-1}]$). One can easily show that the derivations
     $D, d\frac{\partial}{\partial c_i}, i=1, \ldots ,n-1 $ on the domain  $A$ are linearly independent over  $R$ and form a basis of the vector space  $R C_{\Der A}(D)$ over the ring      $R = \mathrm{Frac}(A)$. The latter means that  $C_{\Der A}(D)$ has rank  $n$ over  $A.$

     Now let  $D$ do not have any slices in $A.$ Then, by Lemma  \ref{slice}, there exists a local slice  $h_1 \in A$ such that  $D(h_1) = h_0 \ne 0$ and
     $s =h_1/h_0$ is a slice for the extension  $\overline{D}$ of the derivation      $D$ on the ring of fractions  $A[h_0^{-1}].$

     Every derivation  $T$ of the ring  $C = {\rm Ker}D$ can be uniquly lifted to the derivation  $\overline{T}$ of the ring  $A[h_0^{-1}]$ by putting  $\overline{T}(s) = 0$
     (because  $A[h_0^{-1}] = C[h_0^{-1}][s]$ and the case of slice was considered above). It is easy to check that
     $[\overline{T}, \overline{D}] = 0$, where $\overline{D}$ is the extension of
     $D$ on the ring  $A[h_0^{-1}]$
     (clearly, $\overline D$ is a locally nilpotent derivation and      $\Ker \overline{D} = C[h_0^{-1}]$). For the system of generators
     $a_1,\dots,a_k$ of the algebra $A$ we have
     $\overline{T}(a_i) = {b_i}/(h_0^{k_i})$ for some
     $b_i \in A$ and $k_i \ge 0.$
     Denote  $k = \max k_i$. Then  $h_0^k \overline{T}$ commutes with      $\overline{D}$ and the inclusion
     $h_0^k \overline{T}(A) \subseteq A$ holds.
     As the multiplication of derivations from $C_{\Der A}(D)$ by the powers of $h_0$ does not change the rank of $C_{\Der A}(D)$ over $R = \mathrm{Frac}(A)$, we can repeat considerations from the case when there exists a slice of $D$ in $A$ and obtain the desired equality $\dim_R RC_{\Der A}(D) = n$.

\end{proof}

\end{document}